\documentclass[11pt]{article}

\usepackage{booktabs}

\usepackage{algorithm}
\usepackage{algorithmic}

\usepackage[USenglish]{babel}
\usepackage[rgb]{xcolor}
\usepackage{subcaption}
\usepackage{nicefrac}

\usepackage[table]{xcolor}

\usepackage[hyphens]{url}

\usepackage{fullpage}

\usepackage{tikz}
\usetikzlibrary{shapes.geometric}
\usetikzlibrary{fit}
\usetikzlibrary{arrows}
\usetikzlibrary{patterns}
\usetikzlibrary{matrix}
\usetikzlibrary{positioning}
\usetikzlibrary{backgrounds}
\usetikzlibrary{shadows}
\usetikzlibrary{patterns, ,decorations.pathmorphing}
\usetikzlibrary{positioning,arrows.meta,shadows.blur}

\pgfdeclarelayer{background}
\pgfdeclarelayer{main}
\pgfdeclarelayer{foreground}
\pgfsetlayers{background,main,foreground}

\usepackage{pgfplots}
\pgfplotsset{compat=1.18}

\usepackage[margin=1in]{geometry}
\usepackage{setspace}
\usepackage{csquotes}
\usepackage{microtype}
\usepackage{amsmath, amssymb,amsthm,thmtools,thm-restate}

\usepackage{mathabx}
\DeclareMathOperator{\E}{\mathbb{E}}

\usepackage{ifxetex}
\ifxetex
\usepackage[bigdelims,vvarbb]{newpxmath}
\usepackage{mathspec}
\setallmainfonts(Digits,Latin){Noto Serif}
\else
\usepackage{mathpazo}
\fi

\usepackage{bm}
\usepackage{enumitem}
\usepackage[colorlinks=true,allcolors=linkColor, hypertexnames=false]{hyperref} 
\usepackage{nameref}
\usepackage[nameinlink]{cleveref}
\usepackage{stmaryrd}
\usepackage[round]{natbib}

\newcommand{\citef}[1]{\citeauthor{#1}~\citeyear{#1}}
\def\hy{\hbox{-}\nobreak\hskip0pt}

\DeclareMathSymbol{\mlq}{\mathord}{operators}{``}
\DeclareMathSymbol{\mrq}{\mathord}{operators}{`'}

\newcommand{\antipodal}[1]{\bar{#1}}
\newcommand{\dist}[2]{d(#1,#2)}
\newcommand{\fprime}{\hat{f}}

\renewcommand\epsilon\varepsilon
\renewcommand\emptyset\varnothing
\newcommand\Crestrict[2]{{
  \left.\kern-\nulldelimiterspace{}
  #1 
  \right|_{#2} 
  }}

\definecolor{color1}{Hsb}{100, 0.8, 0.8}
\definecolor{color2}{Hsb}{220, 0.8, 0.8}
\definecolor{color3}{Hsb}{340 , 0.8, 0.8}
\definecolor{linkColor}{Hsb}{11, 0.7, 0.6}

\declaretheorem[shaded={bgcolor=color2!10}]{theorem}
\declaretheorem[sibling=theorem, shaded={bgcolor=color1!10}]{lemma, corollary, proposition, conjecture}

\theoremstyle{definition}

\theoremstyle{remark}

\DeclareUnicodeCharacter{0301}{\'{e}}

\newcommand{\lexprec}{\prec_{\text{lex}}}
\title{
From the Finite to the Infinite:\\ Sharper Asymptotic Bounds on Norin's Conjecture 
via SAT }
\author{Markus Kirchweger \\ TU Wien \\ \url{mk@ac.tuwien.ac.at} \and Tomáš Peitl \\ TU Wien\\ \url{peitl@ac.tuwien.ac.at}  \and  Bernardo Subercaseaux \\ Carnegie Mellon University \\ \url{bersub@cmu.edu}
\and Stefan Szeider \\ TU Wien \\ \url{sz@ac.tuwien.ac.at}}

\begin{document}
\maketitle
\begin{abstract}
    Norin (2008) conjectured that any $2$-edge-coloring of the hypercube $Q_n$ in which antipodal edges receive different colors must contain a monochromatic path between some pair of antipodal vertices. While the general conjecture remains elusive, progress thus far has been made on two fronts: finite cases and asymptotic relaxations. The best finite results are due to Frankston and Scheinerman (2024) who verified the conjecture for $n \leq 7$ using SAT solvers, and the best asymptotic result is due to Dvořák (2020), who showed that every $2$-edge-coloring of $Q_n$ admits an antipodal path of length $n$ with at most $0.375n + o(n)$ color changes. We improve on both fronts via SAT.  First, we extend the verification to $n = 8$ by introducing a more compact and efficient SAT encoding, enhanced with symmetry breaking and cube-and-conquer parallelism. The versatility of this new encoding allows us to recast parts of Dvořák's asymptotic approach as a SAT problem, thereby improving the asymptotic upper bound to $0.3125n + O(1)$ color changes. Our work demonstrates how SAT-based methods can yield not only finite-case confirmations but also asymptotic progress on combinatorial conjectures.
\end{abstract}

\section{Introduction}



While solvers for propositional satisfiability (SAT) have undergone massive improvements in the past 25 years~\citep{FichteLeberreHecherSzeider23},
it is not obvious how to best leverage their computational power for advancing mathematical progress.
Back in 1948, the great logician Alfred Tarski wrote regarding a geometry problem that
\emph{``the machine would permit us to test a hypothesis for any special value of $n$. We could carry out such tests for a sequence of consecutive values $n=2,3,..$ up to, say, $n=100$. If the result of at least one test were negative, the hypothesis would prove to be false; otherwise, our confidence in the hypothesis would increase, and we should feel encouraged to attempt establishing the hypothesis''}~\citep{R109}. 
SAT solvers have largely delivered on Tarski's ideal, allowing e.g., the resolution of the Boolean Pythagorean triples problem~\citep{HeuleKullmannMarek16}, which required examining colorings of the positive integers up to $n = 7825$, or similarly, the resolution of Keller's conjecture on dimension $8$ via SAT which required discarding cliques of size $128$ on graphs of up to $7$ million vertices~\citep{brakensiekResolutionKellersConjecture2020}. Even closer to the example that motivated Tarski, by using a clever SAT encoding, \citet{heuleHappyEndingEmpty2024} proved that every $30$ points in the plane, without three of them collinear, must contain an empty convex hexagon.

While the list of contributions of SAT to different areas of mathematics, especially discrete, keeps growing (e.g., the website \href{https://sat4math.com}{sat4math.com} compiles roughly a hundred papers on the topic), there are remarkably few examples in which SAT solving allows for improving asymptotic bounds~\citep{kulikovImprovingCircuitSize2018, capSets}.




In this paper, we take the usage of SAT for mathematics further,  focusing on a graph theory conjecture of~\citet{sfuEdgeantipodalColorings} concerning monochromatic paths in edge colorings of the $n$-dimensional hypercube.
We obtain classical-style finite results, but we also go beyond what would normally be expected of a SAT-based method and obtain highly non-trivial asymptotic results.
In order to achieve both, we design novel, and mainly in the second case quite creative, SAT encodings.
Our contributions are twofold:

\begin{enumerate}
	\item
		\label{item:contrib-finite}
		We develop a novel encoding and symmetry breaking that improves on previous encodings and approaches to solving the conjecture.
		With this, we confirm~\Cref{conj:norine-geodesic} (see~\Cref{sec:conjectures}) in the $8$\hy dimensional case, up from the previously best known value~$7$.
		Notably, while~\citet{frankston2024provingnorinesconjectureholds} estimated that solving $n=8$ would take 57.3 CPU years, our (improved) encoding only needed five CPU \emph{days}, for a speedup of more than 3 orders of magnitude ($4000\times$).
	\item
		\label{item:contrib-asymp}
		We obtain improved asymptotic bounds on a quantitative version of the conjecture.
		We achieve this by analyzing a proof of~\citet{dvorakNoteNorinesAntipodalColouring2020}, which uses facts about the $3$\hy dimensional case in order to derive bounds for arbitrary dimension~$n$.
		We extend our SAT encoding from~\Cref{item:contrib-finite} and with it determine the corresponding facts for dimension $6$, which allows us to improve the conclusion of \citeauthor{dvorakNoteNorinesAntipodalColouring2020}'s argument.
		This is a powerful hybrid general theorem-proving strategy: the high-level proof structure is human-designed and can use relatively advanced mathematics (induction, probabilistic reasoning, etc.), but the required \emph{case analysis} (which \citeauthor{dvorakNoteNorinesAntipodalColouring2020} had to verify by hand for $n=3$) is done automatically by the SAT solver.
		The case distinction we need to handle involves both counting and averaging values, thereby leading to a rather creative use of SAT not typically seen in applications.
\end{enumerate}

We see this work as a qualitative leap compared even to SAT\hy based resolution of seemingly infinite conjectures.
In those previous cases it so happened that the apparent infinitude could be compressed into a finite statement.
In our case, we are using finite bits of information as plug-in into a proof that covers an infinite number of cases.

In~\Cref{sec:conjectures}, we explain the conjectures that we are addressing and our contributions in detail.
Afterwards, we discuss the finite encoding and other automated-reasoning techniques used in the computation (\Cref{sec:enc}).  Then,~\Cref{sec:asymptotics} presents the asymptotic approach, and~\Cref{sec:experiments} details computational experiments. We conclude and discuss future work in~\Cref{sec:conclusion}.
Our code is publicly available at~\url{https://github.com/bsubercaseaux/norine}.

\subsection{Conjectures and Contributions}
\label{sec:conjectures}

The hypercube graph of order $n$  denoted $Q_n$, has vertex set $\{0, 1\}^n$ and undirected edges between vertices that differ on exactly one coordinate. We will assume throughout that $n \geq 2$.
For a vertex $v = (v_1, \ldots, v_n)$, its \emph{antipodal vertex} is $\antipodal{v} = (1-v_1, \ldots, 1-v_n)$. That is, $\bar{v}$ is the unique vertex at distance $n$ from $v$.
An \emph{antipodal path} is a path between a vertex $v$ and its antipodal $\bar{v}$.
A \emph{geodesic} on $Q_n$ is a shortest path between two vertices.
An \emph{antipodal geodesic} is an antipodal path of length $n$.
The \emph{antipodal edge} of an edge $\{u, v\}$ is simply $\{\bar{u}, \bar{v}\},$ and thus a $2$-coloring of the edges of $Q_n$ is said to be antipodal if every edge gets a different color from its antipodal edge.

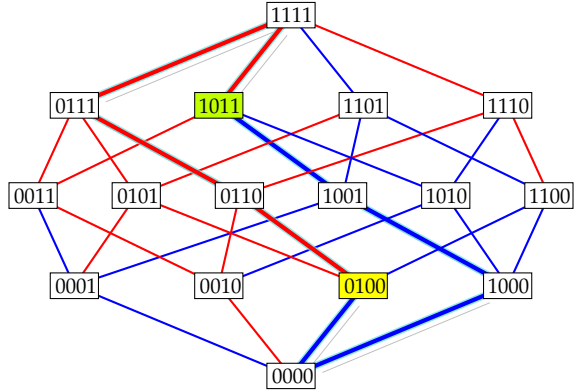
\begin{figure}
    \centering
    \begin{tikzpicture}[scale=0.6]
     \tikzset{every node/.style={  text width=1.3em, inner sep=2pt, text centered, font=\scriptsize}}
\tikzset{
  halo/.style={
    preaction={
      draw=lightgray!50!cyan,
      line width=2.7pt,
      opacity=0.7,
      drop shadow,
    },
  }
}


\tikzset{
ee/.style={
        shorten <=-5pt,
        shorten >=-5pt,
    }
}

\tikzset{
shadowed/.style={preaction={transform canvas={shift={(0.8pt,-0.5pt)}},draw=gray,very thick, opacity=.3}}
}

\node[fill=white, draw] (0000) at (8.0,0) {0000};
\node[fill=white, draw] (0001) at (3.2,2) {0001};
\node[fill=white, draw] (0010) at (6.4,2) {0010};
\node[fill=yellow, draw] (0100) at (9.6,2) {0100};
\node[fill=white, draw] (1000) at (12.8,2) {1000};
\node[fill=white, draw] (0011) at (2.2857142857142856,4) {0011};
\node[fill=white, draw] (0101) at (4.571428571428571,4) {0101};
\node[fill=white, draw] (0110) at (6.857142857142857,4) {0110};
\node[fill=white, draw] (1001) at (9.142857142857142,4) {1001};
\node[fill=white, draw] (1010) at (11.428571428571429,4) {1010};
\node[fill=white, draw] (1100) at (13.714285714285714,4) {1100};
\node[fill=white, draw] (0111) at (3.2,6) {0111};
\node[fill=lime, draw] (1011) at (6.4,6) {1011};
\node[fill=white, draw] (1101) at (9.6,6) {1101};
\node[fill=white, draw] (1110) at (12.8,6) {1110};
\node[fill=white, draw] (1111) at (8.0,8) {1111};
\begin{pgfonlayer}{background}
\draw[opacity=1,  thick, blue, ee] (0000) -- (0001);
\draw[opacity=1,  thick, red, ee] (0000) -- (0010);
\path[draw, opacity=1, ultra thick, blue,  ee,  halo] (0000) -- (0100);
\path[draw, opacity=1,  ultra thick, blue, ee,  halo] (0000) -- (1000);
\path[draw, opacity=1,  ultra thick, blue, ee, halo] (1001) -- (1011);
\path[draw, opacity=1,  ultra thick, blue, ee, halo] (1000) -- (1001);
\draw[opacity=1,  thick, blue, ee] (0001) -- (0011);
\draw[opacity=1,  thick, red, ee] (0001) -- (0101);
\draw[opacity=1,  thick, blue, ee] (0001) -- (1001);
\draw[opacity=1,  thick, red, ee] (0010) -- (0011);
\draw[opacity=1,  thick, red, ee] (0010) -- (0110);
\draw[opacity=1,  thick, blue, ee] (0010) -- (1010);
\draw[opacity=1,  thick, red, ee] (0100) -- (0101);
\draw[opacity=1,  ultra thick, red, halo, ee] (0100) -- (0110);
\draw[opacity=1,  thick, blue, ee] (0100) -- (1100);

\draw[opacity=1,  thick, blue, ee] (1000) -- (1010);
\draw[opacity=1,  thick, blue, ee] (1000) -- (1100);
\draw[opacity=1,  thick, red, ee] (0011) -- (0111);
\draw[opacity=1,  thick, red, ee] (0011) -- (1011);
\draw[opacity=1,  thick, red, ee] (0101) -- (0111);
\draw[opacity=1,  thick, red, ee] (0101) -- (1101);

\draw[opacity=1,  thick, blue, ee] (1001) -- (1101);
\draw[opacity=1,  ultra thick, red, halo, ee] (0110) -- (0111);
\draw[opacity=1,  thick, red, ee] (0110) -- (1110);
\draw[opacity=1,  thick, blue, ee] (1010) -- (1011);
\draw[opacity=1,  thick, blue, ee] (1010) -- (1110);
\draw[opacity=1,  thick, blue, ee] (1100) -- (1101);
\draw[opacity=1,  thick, red, ee] (1100) -- (1110);
\draw[opacity=1,  ultra thick, red, halo, ee] (0111) -- (1111);
\draw[opacity=1,  ultra thick, red, halo, ee] (1011) -- (1111);
\draw[opacity=1,  thick, blue, ee] (1101) -- (1111);
\draw[opacity=1,  thick, red, ee] (1110) -- (1111);
\end{pgfonlayer}
    \end{tikzpicture}
    \caption{
		An antipodal coloring of $Q_4$, with monochromatic geodesics between the antipodal vertices $(1,0,1,1)$ and $(0, 1, 0, 0)$ highlighted.
	}
    \label{fig:intro_example}
\end{figure}

\begin{conjecture}[\citef{sfuEdgeantipodalColorings}]\label{conj:norine}
    Any antipodal $2$-coloring of the edges of $Q_n$ has a monochromatic antipodal path.
\end{conjecture}
An example is presented in~\Cref{fig:intro_example}.
The cases $n \leq 5$ were covered by~\citet{federHypercubeLabellingsAntipodal2013}, and the case $n = 6$ was proved first via SAT by~\citet{zulkoskiCombiningSATSolvers2017a}, and then manually by~\citet{West_Wise_2019}.
Recently Frankston and Scheinerman proved the case $n = 7$, also via SAT.
The two previous SAT approaches in fact proved a stronger variant of the conjecture:

\begin{conjecture}[Geodesic Norin]\label{conj:norine-geodesic}
    Any antipodal $2$-coloring of the edges of $Q_n$ has a monochromatic antipodal geodesic.
\end{conjecture}

\citet{federHypercubeLabellingsAntipodal2013} posed the following variant:

\begin{conjecture}[\citef{federHypercubeLabellingsAntipodal2013}]\label{conj:norine-allcolorings}
    Any $2$-coloring of the edges of $Q_n$ has an antipodal path $P$ with at most one \emph{``color change.''} That is, $P$ is the concatenation of two monochromatic paths.
\end{conjecture}

Note that~\Cref{conj:norine-allcolorings} is not restricted to antipodal colorings and~\Cref{conj:norine-allcolorings} implies~\Cref{conj:norine}: if $P_1$ and  $P_2$ are monochromatic paths then $\overline{P_2}P_1$ is a monochromatic path in an antipodal coloring, where $\overline{P_2}$ is the antipodal path~\citep{dvorakNoteNorinesAntipodalColouring2020}. 
Finally, the following variant was considered by~\citet{LEADER201429}.

\begin{conjecture}\label{conj:norine-allcolorings-geodesic}
    Any $2$-coloring of the edges of $Q_n$ has an antipodal geodesic $P$ with at most one color change.
\end{conjecture}

As with the relationship between~\Cref{conj:norine-allcolorings} and ~\Cref{conj:norine}, \Cref{conj:norine-allcolorings-geodesic} implies~\Cref{conj:norine-geodesic}. In this case, the converse also holds~\citep[Proposition 4.6]{LEADER201429}. More precisely, if~\Cref{conj:norine-geodesic} holds for $n+1$, then~\Cref{conj:norine-allcolorings-geodesic} follows for $n$.
\Cref{fig:conjectures} summarizes the relation between the different conjectures.


\begin{figure}
    \centering
    \tikzset{ 
  colorful steps2/.style={ 
    path picture={
       \foreach \c [count=\i from 0] in {#1} { 
          \fill[\c, opacity=0.8] let \p1=($(path picture bounding box.east)-(path picture bounding box.west)$) in ([xshift={(1/2)*\i*\x1}]path picture bounding box.south west) rectangle (path picture bounding box.north east); 
      } 
    } 
  } 
}

\tikzset{ 
  colorful steps9/.style={ 
    path picture={
       \foreach \c [count=\i from 0] in {#1} { 
          \fill[\c, opacity=0.2] let \p1=($(path picture bounding box.east)-(path picture bounding box.west)$) in ([xshift={(1/7)*\i*\x1}]path picture bounding box.south west) rectangle (path picture bounding box.north east); 
      } 
    } 
  } 
}

\begin{tikzpicture}

\def\halfDist{0.5}

\node[draw, fill=white, align=center] (1) at (0,0)  {\textbf{Conjecture 1} \\ \\ {\footnotesize \emph{Coloring}: antipodal} \\ {\footnotesize \emph{Path}: arbitrary, monochromatic}};

\node[draw, fill = white, circle, inner sep=1pt] (a) at (-1.5,0.2) {$u$};
\node[draw, fill= black, text=white, circle, inner sep=1pt] (b) at (1.5,0.2) {$\bar{u}$};

\draw[thick,color=red,decorate,decoration={random steps, segment length=4pt, amplitude=2pt},line cap=round] (a) -- (0.0,0.3) -- (0.2, -0.1) -- (-0.2, 0.15) -- (0.4, 0.4) -- (b);

\node[draw, fill=white, align=center] (2) at (5.5+\halfDist, 0) {\textbf{Conjecture 2} \\  \\ {\footnotesize  \emph{Coloring}: antipodal} \\ {\footnotesize  \emph{Path}: geodesic, monochromatic}};

\node[draw, fill = white, circle, inner sep=1pt] (a2) at (4.0+\halfDist,0.2) {$u$};
\node[draw,  fill= black, text=white, circle, inner sep=1pt] (b2) at (7.0+\halfDist,0.2) {$\bar{u}$};

\draw[thick,color=red,decorate,decoration={random steps, segment length=4pt, amplitude=2pt},line cap=round] (a2) -- (b2);

\node[draw, fill=white, align=center] (3) at (0, -3-\halfDist) {\textbf{Conjecture 3} \\ \\ {\footnotesize  \emph{Coloring}: arbitrary} \\ {\footnotesize  \emph{Path}: arbitrary, $\leq \! 1$ color change}};

\node[draw, fill = white, circle, inner sep=1pt] (a3) at (-1.5,-2.8-\halfDist) {$u$};
\node[draw,  fill= black, text=white, circle, inner sep=1pt] (b3) at (1.5,-2.8-\halfDist) {$\bar{u}$};

\draw[thick,color=red,decorate,decoration={random steps, segment length=4pt, amplitude=2pt},line cap=round] (a3) -- (0.0, 0.3-3.0-\halfDist) -- (0.2, -0.1-3.0-\halfDist);
\draw[thick,color=blue,decorate,decoration={random steps, segment length=4pt, amplitude=2pt},line cap=round] (0.2, -0.1-3.0-\halfDist) -- (-0.2, 0.15-3.0-\halfDist) -- (0.4, 0.4-3.0-\halfDist) -- (b3);

\node[draw, fill=white, align=center] (4) at (5.5+\halfDist, -3-\halfDist) {\textbf{Conjecture 4} \\ \\ {\footnotesize  \emph{Coloring}: arbitrary} \\ {\footnotesize \emph{Path}: geodesic, $\leq \! 1$ color change}};

\node[draw, fill = white, circle, inner sep=1pt] (a4) at (-1.5+5.5+\halfDist,-2.8-\halfDist) {$u$};
\node[draw,  fill= black, text=white, circle, inner sep=1pt] (b4) at (1.5+5.5+\halfDist,-2.8-\halfDist) {$\bar{u}$};

\draw[thick,color=blue,decorate,decoration={random steps, segment length=4pt, amplitude=2pt},line cap=round] (a4) -- (5.5, -2.8-\halfDist);
\draw[thick,color=red,decorate,decoration={random steps, segment length=4pt, amplitude=2pt},line cap=round] (5.5, -2.8-\halfDist) -- (b4);

\draw[very thick, black, -latex, double]  (2) -- (1);

\draw[very thick, black, -latex, double]  (3) -- (1);

\draw[very thick, black, latex-latex, double]  (2) -- (4);

\node[align=center] (label) at (6.95+3.2*\halfDist, -1.5-0.5*\halfDist) {

\footnotesize \citet[Prop. 4.6]{LEADER201429} };

\node[] (label2) at (2.75+0.6*\halfDist, 0.25) {\footnotesize trivial };

\node[] (label3) at (2.75+0.6*\halfDist, -2.75-\halfDist) {\footnotesize trivial };

\node[align=center] (label4) at (-1.2-2*\halfDist, -1.5-0.5*\halfDist) {\footnotesize  
\citet[Thm. 5]{federHypercubeLabellingsAntipodal2013}
};

\draw[very thick, black, -latex, double]  (4) -- (3);


\end{tikzpicture}
    \caption{Overview of conjectures and their implications. Recall that~\Cref{conj:norine-geodesic} holding for $n$ implies~\Cref{conj:norine-allcolorings-geodesic} for $n-1$.}
    \label{fig:conjectures}
\end{figure}
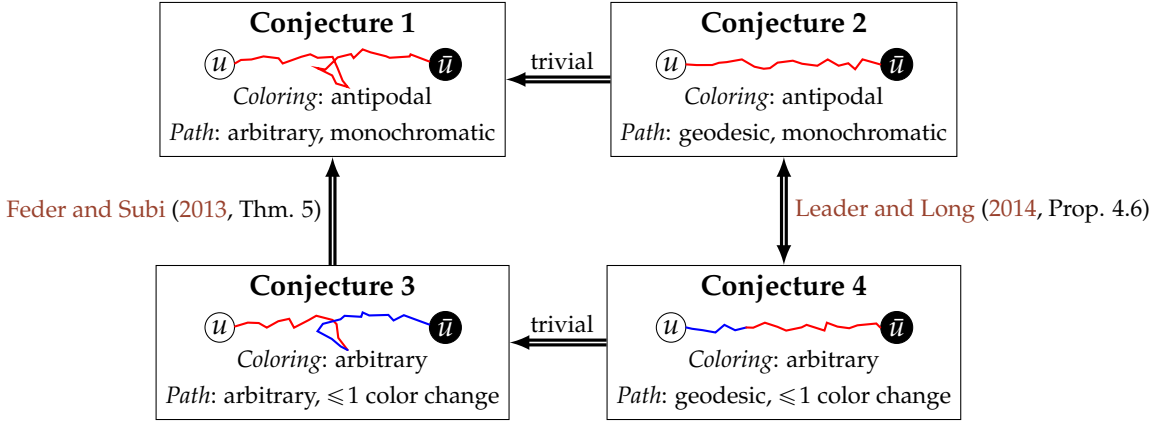

We can now formally state our main results.

\begin{theorem}\label{thm:verification_8}
    \Cref{conj:norine} and~ \Cref{conj:norine-geodesic} hold for $n=8$.  \Cref{conj:norine-allcolorings}, and  \Cref{conj:norine-allcolorings-geodesic} hold for $n = 7$.
\end{theorem}

\begin{theorem}\label{thm:asymptotic}
    For $n \geq 2$, and any 2-coloring of the edges of $Q_n$, there is an antipodal geodesic with at most $0.3125 n + 6$ color changes.
\end{theorem}

\Cref{thm:asymptotic} concerns a quantitative version of~\Cref{conj:norine-allcolorings}, and even though~\Cref{conj:norine-allcolorings} requests at most 1 color change, it is a tantalizing open problem to prove the existence of an antipodal path with $o(n)$ color changes~\citep{soltesz1switchConjecture2017}.

\section{Finite Encoding} \label{sec:enc}
For basics on satisfiability and encodings, we direct the reader to~\citet{BiereHeuleMaarenWalsh09}.
In this section we describe our encodings searching for 2-colorings of the edges of $Q_n$ that would constitute counterexamples to~\Cref{conj:norine}~or~\Cref{conj:norine-geodesic}.
Concretely, for each value of $n$, we build CNF formulas $\Phi_n$ and $\Psi_n$ such that:
\begin{itemize}
    \item $\Phi_n$ is unsatisfiable if, and only if,~\Cref{conj:norine} holds for $n$.
    \item $\Psi_n$ is unsatisfiable if, and only if,~\Cref{conj:norine-geodesic} holds for $n$.
\end{itemize}

Previous work either used propagators to avoid monochromatic geodesics~\citep{zulkoskiCombiningSATSolvers2017a} or introduced a clause for each antipodal geodesic~\citep{frankston2024provingnorinesconjectureholds}, resulting in a large encoding. In turn, we will now present encodings that are much more efficient both in theory and in practice. Since the constructions of $\Phi_n$ and $\Psi_n$ are almost identical, we will present a single construction and specify the only constraint that changes depending on whether we consider $\Phi_n$ or $\Psi_n$.

A first observation to improve upon the encoding of~\citet{frankston2024provingnorinesconjectureholds} is that while they encode antipodal colorings explicitly (saying an edge $\{u,v\}$ is red iff $\{\bar{u}, \bar{v}\}$ is blue), we do this implicitly by focusing on the ``lexicographically smaller'' half of the graph. 
For any two distinct sequences $s = (s_1, s_2, \ldots, s_n)$ and $t = (t_1, t_2, \ldots, t_n)$, we say $s \prec_{\text{lex}} t$ (i.e., $s$ is lexicographically smaller than $t$) if at the first index $k \in \{1, \ldots, n\}$ where $s_k \ne t_k$, we have $s_k < t_k$.
We use variables $r_{u, v}$ for $\{u,v\} \in E(Q_n)$ for $u \lexprec v$ to indicate whether a red edge is present from $u$ to $v$. By abuse of notation, we write \( r_{u,v} \) for all \( u,v \in V(Q_n) \), with the understanding that it refers to 
$ r_{\min_{\lexprec}(u,v),\, \max_{\lexprec}(u,v)} $.
For antipodal colorings, we omit variables $r_{u,v}$ if $\antipodal{v} \lexprec u$ by replacing them with $\lnot r_{\antipodal{v},\antipodal{u}}$.

In addition, we have variables $p_{u, v}$ for $u \lexprec v$ and $u \lexprec \antipodal{u}$ to indicate whether there is a red path (in case of $\Phi_n$, for~\Cref{conj:norine}) or a red geodesic (in case of $\Psi_n$, for~\Cref{conj:norine-geodesic}) from $u$ to $v$.\footnote{Note that we only require the variable to be true if there is a monochromatic path but not the other direction.}
We have the following clauses:
\begin{equation}\label{eq:constraint1}
    r_{u, v} \rightarrow p_{u, v}
\end{equation}
for $u \in V(Q_n), v \in N(u), u \lexprec \antipodal{u}$ and 
\begin{equation}\label{eq:constraint2}
    (p_{u, v} \land r_{v, w}) \rightarrow p_{u, w}
\end{equation}
for $u,v,w \in V(Q_n), w \notin N(u), w \in N(v), u \lexprec \antipodal{u}$. 
 We use
$\dist{u}{v} 
$ to denote the \emph{distance} between two vertices.
If we want to restrict our paths to geodesics then we have to further restrict the condition to $\dist{u}{v} + 1 = \dist{u}{w}$, i.e., the distance to vertex $u$ increases. 
To encode the absence of monochromatic paths/geodesics, we add the unit clauses  $\lnot p_{v, \bar{v}}$ for  $v \in V(Q_n), v \lexprec \antipodal{v}$.

Our constraints~\eqref{eq:constraint1} and~\eqref{eq:constraint2} add up to at most $2^{n-1} \cdot 2^n \cdot n$ clauses, which is polynomial in $|V(Q_n)|$. In contrast, explicitly using a single clause for each antipodal geodesic results in $2^{n - 1} \cdot n!$ clauses. 
\Cref{tab:table_encodings} shows the number of variables and clauses for the different encodings: for $n\geq 6$ our encodings use significantly fewer clauses.




\definecolor{bestbg}{HTML}{DDF3B0} 
\newcommand{\best}[1]{\cellcolor{bestbg}\textbf{#1}}
\begin{table}[ht]
\caption{Comparison of our encodings vs.~\citeauthor{frankston2024provingnorinesconjectureholds}. } 
    \centering
    \begin{tabular}{@{}r|r@{~~}r@{~~}r|r@{~~}r@{~~}r@{}}
        \toprule
         $n$ &  \multicolumn{3}{c|}{$\#$Variables}
 & \multicolumn{3}{c}{$\#$Clauses}
 \\
 & F\&S & $\Phi_n$ & $\Psi_n$  & F\&S & $\Phi_n$ & $\Psi_n$ \\ \midrule 
4 & \best{32} & 776 & 760 & \best{227} & 2.4K & 2.1K  \\
5 & \best{80} & 2.2K & 2.2K & \best{2.0K} & 8.1K & 6.2K  \\
6 & \best{192} & 6.5K & 6.5K & 23.2K & 30.4K & \best{20.0K}  \\
7 & \best{448} & 21.4K & 21.3K & 323.0K & 125.5K & \best{73.5K}  \\
8 & \best{1.0K} & 76.2K & 76.0K & 5.2M & 544.7K & \best{296.9K}  \\
9 & \best{2.3K} & 285.6K & 285.1K & 92.9M & 2.4M & \best{1.3M}  \\
         \bottomrule
    \end{tabular}
    
\label{tab:table_encodings}
\end{table}


\subsection{Symmetry breaking}\label{subsec:symmetry-breaking}

An important aspect of combinatorial search via SAT solvers is breaking symmetries, i.e., avoiding isomorphic copies in the search space \citep[e.g.,][]{CodishMillerProsserStuckey19,KirchwegerSzeider24}. 

\citet{zulkoskiCombiningSATSolvers2017a} opted to exploit symmetries by learning some symmetric version of a geodesic whenever the propagator encounters a monochromatic antipodal geodesic. 
    \citet{frankston2024provingnorinesconjectureholds} encoded that the first vertex has red edges to its first $\lfloor n/2 \rfloor$ neighbors, and a blue edge to the next one.

We opt for a more thorough symmetry breaking by adding lex-leader constraints~\citep{CrawfordGinsbergLuksRoy96}.
Given two sequences of literals $(x_1, \ldots, x_m)$ and $(y_1, \ldots, y_m)$ a \emph{lex-leader constraint} ensures that $(\phi(x_1), \ldots, \phi(x_m)) \preceq_{\text{lex}} (\phi(y_1), \ldots, \phi(y_m))$ for each model $\phi$.
These constraints can be used to prune the search space while preserving satisfiability.


Let $\pi: \![n] \!\to \![n]$ be a permutation, $f: \![n]\! \to \!\{0,1\}$, and consider the symmetry $S_{\pi,f}: V(Q_n) \to V(Q_n)$ defined coordinate-wise as
\[
 \left(S_{\pi,f}(v)\right)_i=  v_{\pi(i)} + f(i) \bmod 2.
\]
The permutation $\pi$ reorders the dimensions, while the function $f$ corresponds to ``flipping'' a dimension, i.e., swapping the role of $0$ and $1$ in that dimension. We use $\mathcal{H}_n$ to denote the set of all symmetries of $Q_n$. This is also known as the \emph{hyperoctahedral group}. This group $\mathcal{H}_n$ acts on the edges of $Q_n$ in a natural way, i.e., $S_{\pi,f}( \{u,v\}) = \{S_{\pi,f}(u), S_{\pi,f}(v)\}$.

Then, if $c: E(Q_n) \to \{0,1\}$ is a counterexample to~\Cref{conj:norine}, then the coloring
$c'(e) = c(S_{\pi,f}(e)) $ is also a counterexample.
 We want to avoid the solver having to refute all these \emph{isomorphic} colorings separately.

Let $e_1, e_2 \ldots, e_m$ be a fixed ordering of the edges of $Q_n$.  We want to restrict our search for colorings $c$ which are minimal by this ordering, i.e., 
\[
     (c(e_1),  \ldots, c(e_m)) \preceq_{\text{lex}} ( c(S_{\pi,f}(e_1)), \ldots, c(S_{\pi,f}(e_m)))
\]
for all $S_{\pi,f} \in \mathcal{H}_n$.
In principle,  we could add a lex-leader constraint for each symmetry $S_{\pi,f} \in \mathcal{H}_n$. The problem with this approach is the number of symmetries; there are $n!\cdot 2^n$, which would be a bottleneck of our encoding.
Similarly to symmetry-breaking in graph-search problems by \citet{CodishMillerProsserStuckey19}, we only consider a subset of these symmetries which are ``simple.''
We break symmetries stemming from $S_{\pi,f}$ where $\pi$ is a transposition (exactly two elements are swapped) and $f$ maps to $1$ for at most one value.

Since there are $n2^{n-1}$ edges, the number of clauses of a lex-leader constraint is not negligible even for a single symmetry.
Based on a parameter \texttt{max\_comp}, we aim to shorten the maximum number of comparisons between two sequences as follows.
First, for a given symmetry $S_{\pi,f}$, we simplify the edge sequence by removing fixpoints (i.e., elements $e$ with $S_{\pi,f}(e) = e$). Next, we only keep the first  \texttt{max\_comp} elements in our edge sequence and only construct the lex-leader constraint based on this shortened sequence.
We run experiments testing different values, and observe that setting \texttt{max\_comp} in the range [10,20] breaks a significant number of symmetries without adding too many clauses (see~\Cref{sec:experiments}). 

By selecting the first $n$ edges in the ordering to be those incident to the vertex $\vec{0} = (0, \ldots, 0)$, we ensure that in any lexicographically minimal coloring, the vertex $\vec{0}$ has the smallest \emph{red degree}. We can enforce this property explicitly using cardinality-constraint variables (i.e., if $d_{v, i}$ denotes that vertex $v$ has red degree $ \geq i$, then $  d_{\vec{0}, i} \to d_{v, i}$ for every other $v$), and such variables can be constructed according to~\citet{Sinz05}. Note that such restrictions typically go beyond what can be enforced using simple transpositions of dimensions, as they often require more complex permutations.

For our experiments, we define an ordering of the edges in the $n$-dimensional discrete hypercube as follows: we begin with the vertex $(0, \ldots, 0)$ and include all edges incident to it. We then iterate over the remaining vertices in lexicographic order, starting from $(0, \ldots, 0, 1)$, and for each vertex, we add its incident edges that have not yet been included in the ordering. This process continues until all edges have been assigned a position in the sequence.

To justify our choice of symmetries for generating lex-leader constraints, we also conducted experiments involving permutations of more dimensions. The results, in~\Cref{sec:experiments}, indicate that the overhead introduced by the increased number of constraints outweighed the benefits gained from breaking a larger set of symmetries.

Finally, we remark that symmetry-breaking constraints for $Q_n$ were studied by~\citet{LPAR2023:Toward_Optimal_Radio_Colorings} in a vertex-coloring problem with $k > n$ colors, which is significantly easier: there one can break the $n!$ symmetries between the dimensions by simply enforcing that if $u$ and $v$ are neighbors of $\vec{0}$, then $\text{color}(u) \leq \text{color}(v) \iff u \preceq_{\text{lex}} v$.

\subsection{Cube And Conquer Parallelism}

As the dimension $n$ increases, the resulting formulas become increasingly challenging for modern SAT solvers. To address this, we use the \emph{Cube And Conquer} (C\&C) approach~\citep{HeuleKullmannWieringaBiere11, HeuleKullmannMarek16,HeuleKullmannBiere18} to enable parallel solving.

C\&C splits the search space into independent subproblems using \emph{cubes}, which are conjunctions of literals. Given a set of cubes $\Gamma = \{\gamma_1, \ldots, \gamma_k\}$ such that $\bigvee_{i =1}^k \gamma_i$ is a tautology, we can replace the original formula $F$ with the logically equivalent formula $F \land \bigvee_{i =1}^k \gamma_i$.
This allows solving each subproblem $F \land \gamma_i$ independently. If any subproblem is satisfiable, then $F$ is satisfiable. If all are unsatisfiable, then $F$ is unsatisfiable.

We aim to generate cubes such that the subproblems $F \land \gamma_i$ have roughly equal difficulty, and each cube requires distinct reasoning to be refuted.   While in some cases doing this effectively has required handcrafted \emph{``splits''}~\citep{heuleHappyEndingEmpty2024, SubercaseauxHeule23}, in our case general-purpose splitting tools have proved effective: \texttt{march\_cu}~\citep{HeuleKullmannWieringaBiere11} and \texttt{proofix}~\citep{proofix}. Both of these can take a depth parameter $d$ and generate roughly $2^d$ cubes; we experimented with $d \in \{9, \ldots, 15\}$, and present results in~\Cref{sec:experiments}.




\section{Asymptotic Improvement}\label{sec:asymptotics}

A natural way to relax~\Cref{conj:norine-allcolorings} is to ask for a path between antipodal vertices that changes colors a \emph{small number of times} (instead of at most once). For any path $P$, its number of color changes, denoted $\gamma(P)$, is the smallest $k$ such that $P$ can be written as the concatenation of $k+1$ monochromatic paths $P_1, \ldots, P_{k+1}$. Equivalently, $\gamma(P)$ is the number of vertices of $P$ that have two edges induced by $P$ with different colors.
 While~\Cref{conj:norine-allcolorings} states that any $2$-edge-coloring of $Q_n$ has an antipodal path with $\gamma(P) \leq 1$, the trivial bound is $\gamma(P) \leq n-1$, and perhaps surprisingly, the best results thus far are still linear in $n$. \citet{LEADER201429} proved that there is always a path $P$ with $\gamma(P) \leq \lfloor \frac{n}{2}\rfloor$, and the state of the art is given by the following result of Dvo\v{r}\'ak.

 \begin{theorem}[\citef{dvorakNoteNorinesAntipodalColouring2020}]\label{thm:dvorak}
     For every $2$-coloring of the edges of $Q_n$, there is an antipodal geodesic $P$ with $\gamma(P) \leq (\nicefrac{3}{8} + o(1))n$.
 \end{theorem}

\begin{figure}
    \centering
        \input{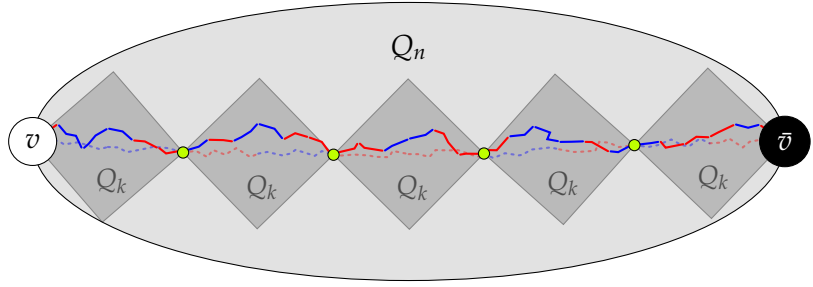}

    \caption{Illustration of the proof for~\Cref{thm:asymptotic}, following the method of~\citet{dvorakNoteNorinesAntipodalColouring2020}. The dotted lines show the improved antipodal geodesic by local changes.}
    \label{fig:proof_illustration}
\end{figure}

We improve upon~\Cref{thm:dvorak} with~\Cref{thm:asymptotic}.
As our improvement builds upon the ideas of Dvo\v{r}\'ak, and the discussion given by~\citet{ReuSlides}, it is worth sketching the high-level idea of Dvo\v{r}\'ak's proof first. The proof uses the probabilistic method, arguing that a randomized algorithm constructs an antipodal geodesic whose expected number of color changes is $(\nicefrac{3}{8} + o(1))n$, and thus there must exist some antipodal geodesic as good as that expected value. 
The algorithm starts by selecting a uniformly random antipodal geodesic $P$, and then partitions $P$ into $n/3$ subgeodesics $P_1, \ldots, P_{n/3}$ of length $3$ each (one can assume that $n$ is a multiple of $3$, since the difference is absorbed by the $o(n)$ term). Note that each subpath $P_i$ is an antipodal path of a subgraph $H_i$ isomorphic to $Q_3$, as illustrated in~\Cref{fig:proof_illustration}. For each $P_i$, the algorithm will keep its first and last vertex, but potentially replace the two intermediate vertices in $H_i$ in order to lower the number of color changes.

Note that, as the starting antipodal geodesic was chosen uniformly at random, it is guaranteed that each vertex has the same probability of being the first vertex in the subgeodesic $P_i$.
To achieve a first simple upper bound on the number of expected color changes of the randomized algorithm, one can take the expected color changes within each subpath $P_i$ after potentially replacing the two intermediate vertices by the best possible options. Additionally, between each pair of consecutive subpaths $P_i, P_{i+1}$ there could be a color change.

The expected color changes within each subpath $P_i$ can be upper bounded by a \emph{worst case} coloring of $Q_3$, i.e., 
\[b := \frac{1}{ |V(Q_3)|} \max_{c \in C} \sum_{u \in V(Q_3)} s_{c, u, \antipodal{u}} \] where $s_{c, u, \antipodal{u}}$ is the minimum number of color changes needed to reach $\antipodal{u}$ with an antipodal geodesic. Using $b$, we derive the following bound on the number of color changes for $Q_n$:
\[ \lfloor  \nicefrac{n}{3} \rfloor\cdot b +  \lfloor \nicefrac{n}{3} \rfloor + 3.\]
The first summand stems from the color changes within the subpaths $P_i$, the second from color changes from $P_i$ to $P_{i+1}$, and the last if $n$ is not divisible by 3. A detailed proof is presented later on, for the stronger statement of~\Cref{thm:main}.

Note that this simple version leads to an asymptotic bound of $(\nicefrac{2}{3} + o(1)) n$.
To achieve a better bound of $(\nicefrac{3}{8} + o(1))n$, Dvo\v{r}\'ak also showed that, for any edge 2-coloring of $Q_n$, a certain fraction of the $Q_3$ subgraphs have a vertex such that all 3 incident edges have the same color. He called these colorings of $Q_3$ \emph{good} and all other colorings \emph{bad}.\footnote{Strictly speaking,~\cite{dvorakNoteNorinesAntipodalColouring2020} called a coloring \emph{good} if one can find four antipodal geodesics with different endpoints, such that these four geodesics
have at most two color changes in total.} For a bad coloring of a $Q_3$ and $v\in V(Q_3)$, there is always a geodesic to $\antipodal{v}$ starting with a red edge with the minimum number of color changes~\citep[Lemma 8]{dvorakNoteNorinesAntipodalColouring2020}. Symmetrically, for bad colorings we can also start with a blue edge without increasing the number of color changes within $Q_3$. In other words, if a subpath $P_i$ ends in a red/blue edge and $P_{i+1}$ is a bad coloring, then we can choose color red/blue as the first edge in $P_{i+1}$, thus avoiding one color change.



Our idea is to partition the random geodesic into larger subpaths (of size $k$ instead of $3$) and use a SAT solver to get an upper bound on the average number of color changes within $Q_k$.
Let $f(k)$ be the average number of color changes of $Q_k$ for the ``worst'' coloring of $Q_k$, i.e., 
\begin{equation}
    \label{eq:fn}
    f(n) := \max_{c \in C} \frac{ \sum_{u \in V} s_{c, u, \antipodal{u}}}{2^n},
\end{equation}
where again $s_{c, u, \antipodal{u}}$ is the minimum number of color changes needed to reach $\antipodal{u}$ with a geodesics.
For a given $n,k \in \mathbb{N}$, we derive the following upper bound on the number of expected color changes:
\( \lfloor \nicefrac{n}{k} \rfloor \cdot f(k) + \lfloor \nicefrac{n}{k} \rfloor + (n \bmod k).\)

We refine this by exploiting that for a coloring $c$ and some vertices within $Q_k$, we can choose the color of the first edge without increasing the number of changes, allowing us to reduce the number of color changes between $P_i$ and $P_{i+1}$ in these cases.
$s^r_{c, u, \antipodal{u}}$ (resp. $s^b_{c, u, \antipodal{u}}$) is the minimum number of color changes needed to reach $\antipodal{u}$ with a geodesic starting with a red (resp. blue) edge, and $s'_{c, u, \antipodal{u}} = \max(s^r_{c, u, \antipodal{u}}, s^b_{c, u, \antipodal{u}})$.
\begin{equation}
    \label{eq:fprimen}
    \fprime (n) := \max_{c \in C} \frac{ \sum_{u \in V} \min (s_{c, u, \antipodal{u}}, s'_{c, u, \antipodal{u}} - 1) }{2^n}.
\end{equation}

We will make use of $\fprime(n)$ to analyze~\Cref{alg:rand},
a randomized algorithm which takes a $2$-coloring of $Q_n$ and an integer $k$ as input and outputs an antipodal geodesic.

\begin{algorithm} 
\caption{Randomized Geodesic with Few Color Changes in $Q_n$} \label{alg:rand}
\begin{algorithmic}[1] 
\REQUIRE Integer $n,k$
\REQUIRE $2$-coloring $c : E(Q_n) \to \{0,1\}$
\ENSURE An antipodal geodesic

\STATE $P \gets$ choose an antipodal geodesic uniformly at random
\STATE $m \gets \left\lfloor \nicefrac{n}{k} \right\rfloor$
\STATE Partition $P=(v_0, \ldots, v_n)$ into $m$ consecutive chunks of length $k$ and one final chunk of length $n \bmod k$ (if nonzero):
\STATE \hspace{\algorithmicindent} $P = P_0 \oplus P_1 \oplus \cdots \oplus P_{m-1} \oplus P_{\text{rem}}$
\STATE \hspace{\algorithmicindent} where $P_i = (v_{ki}, v_{ki+1}, \ldots, v_{k(i+1)})$ for $i < m$
\STATE \hspace{\algorithmicindent} and $P_{\text{rem}} = (v_{km}, \ldots, v_{n})$ if $n \bmod k \neq 0$

\FOR{$i = 0$ to $m-1$}
    \STATE $v_{\text{start}} \gets$ first vertex of $P_i$
    \STATE $v_{\text{end}} \gets$ last vertex of $P_i$
    \STATE Find a geodesic from $v_{\text{start}}$ to $v_{\text{end}}$ with minimal number of color changes under $c$. 
     \IF{$i > 0$}
        \STATE If possible, choose such a geodesic starting with an edge whose color matches the last edge of $P_{i-1}$ (to avoid extra color change)
    \ENDIF
    \STATE Replace $P_i$ in $P$ with this optimized path
\ENDFOR
\RETURN $P$
\end{algorithmic}
\end{algorithm}

\begin{lemma} \label{thm:main}
    Let $n, k \in \mathbb{N}$ with $2 \leq k \leq n$, and let $c$ be a $2$-edge-coloring of the hypercube $Q_n$. Then the expected number of color changes in the antipodal geodesic returned by \Cref{alg:rand} is at most
    \[
        \left\lfloor \nicefrac{n}{k} \right\rfloor \cdot \fprime(k) + \left\lfloor \nicefrac{n}{k} \right\rfloor + (n \bmod k).
    \]
\end{lemma}

\newcommand{\subk}{Q_{\downarrow k}}
\begin{proof}

We focus on bounding the number of color changes between $v_0$ and $v_{k  \left\lfloor \nicefrac{n}{k} \right\rfloor}$, since for the remaining part of the path, we can assume the worst case of having a color change at each vertex, resulting in at most $n \bmod k$ additional color changes.
Let $P' = (v'_0, v'_1, \ldots, v'_{k\left\lfloor \nicefrac{n}{k} \right\rfloor })$ be the geodesic returned by~\Cref{alg:rand} without the remaining part $P_\text{rem}$. Note that $v'_{ik} = v_{ik}$, meaning that the start and endpoint of the subgeodesics do not change.

Let $\subk$ denote the set of all subgraphs of $Q_n$ isomorphic to $Q_k$.
Let us define events $X_{q,i,v}$ for $q \in \subk$, $v \in V(Q_n)$, and $0 \leq i < \left\lfloor \nicefrac{n}{k} \right\rfloor$, where $X_{q,i,v}$ denotes the event that (i) $v = v_{ki}$, (ii) $v_{ki} \in q$, and (iii) $v_{k(i+1)} \in q$. Note that a sub-hypercube of dimension $k$ is uniquely defined by two vertices with distance $k$. We use $H_i$ for the unique sub-hypercube containing $v_{ki},  v_{k(i+1)}$.

Let $\sigma(p)$ denote the number of color changes along a path $p$ in $Q_n$ under the coloring $c$ and $m = \left\lfloor \nicefrac{n}{k} \right\rfloor$. Then the expected number of color changes in the final path $P'$ is:

\begin{align*}
\mathbb{E}\left[\sigma(P')\right]
&=  \mathbb{E}\left[ \sum_{i=0}^{m - 1} \sigma(P_i') + \mathbf{1}_{i > 0} \cdot \sigma(v'_{ki-1}, v_{ki}, v'_{ki+1}) \right] \\
&= \sum_{i=0}^{m - 1} \mathbb{E}\left[ \sigma(P_i') + \mathbf{1}_{i > 0} \cdot \sigma(v'_{ki-1}, v_{ki}, v'_{ki+1}) \right] \\
&= \sum_{i=0}^{m - 1} 
   \sum_{q \in \subk} \sum_{v \in V(Q_n)}  \mathbb{E}\Big[ \sigma(P_i')  + 
   \mathbf{1}_{i > 0} \cdot \sigma(v'_{ki-1}, v_{ki}, v'_{ki+1}) \mid X_{q,i,v} \Big] \cdot \Pr[X_{q,i,v}].
\end{align*}

Now observe that for $v \in V(q)$ (otherwise $\Pr[X_{q,i,v}] = 0$), we have:
\begin{align*}
\Pr[X_{q,i,v}] &= \Pr[q = H_i \land v = v_{ki}] 
\\ &
= \Pr[q = H_i] \cdot \Pr[v = v_{ki} \mid q = H_i] \\ &
= 
\frac{1}{|\subk|} \cdot \frac{1}{2^k},
\end{align*}
since all $q \in \subk$ are equally likely to be chosen as $H_i$, and within each sub-hypercube $q \cong Q_k$, each of the $2^k$ vertices is equally likely to appear as $v_{ki}$.

The number of $k$-dimensional sub-hypercubes is given by
\(
|\subk| = 2^{n-k} \binom{n}{k}.
\)
Substituting this into the expectation, we obtain:
\begin{align*}
\mathbb{E}[\sigma(P')]
&= \frac{1}{2^{n-k} \binom{n}{k} \cdot 2^k} 
\sum_{i=0}^{\nicefrac{n}{k} - 1} 
\sum_{q \in \subk} 
\sum_{v \in V(q)} 
\E\!\left[\sigma(P_i') + \mathbf{1}_{i > 0} \cdot \sigma(v'_{ki-1}, v_{ki}, v'_{ki+1}) \mid X_{q,i,v}\right] \\
&\leq \frac{1}{2^{n-k} \binom{n}{k} \cdot 2^k} 
\sum_{i=0}^{m - 1} 
\sum_{q \in \subk} 
\sum_{v \in V(q)} 
\left( \fprime(k) + 1 \right) \\
&= \frac{1}{2^{n-k} \binom{n}{k} \cdot 2^k} 
\sum_{i=0}^{\nicefrac{n}{k} - 1} 
\sum_{q \in \subk} 
2^k (\fprime(k) + 1) \\
&= \frac{1}{2^{n-k} \binom{n}{k}} 
\sum_{i=0}^{m - 1} 
(\fprime(k) + 1) \cdot |\subk| \\
&= \sum_{i=0}^{m - 1} (\fprime(k) + 1) \\
&= m \cdot (\fprime(k) + 1).
\end{align*}


Adding at most $n \bmod k$ additional color changes for the remainder part concludes the proof.
\end{proof}

Note that the last summand in~\Cref{thm:main} could be replaced by $(\fprime(n \bmod k) + 1)$, but asymptotically this does not make any difference. 


Using~\Cref{thm:main} to improve the upper bound of~\citet{dvorakNoteNorinesAntipodalColouring2020} requires proving for some $k$ that 
\[\frac{n}{k} (\fprime(k) + 1) < \frac{3}{8} \cdot n,\]
or equivalently, that  \(\fprime (k) < \frac{3k}{8} - 1\). Indeed, we achieve this via SAT, obtaining the following result.

\begin{lemma}[Proven computationally]\label{lemma:fprime_6}
    $\fprime(6) = 0.875$.
\end{lemma}

Combining~\Cref{thm:main} with~\Cref{lemma:fprime_6}, our asymptotic bound of~\Cref{thm:asymptotic} follows immediately. 
The encoding used for proving~\Cref{lemma:fprime_6} is detailed in~\Cref{sec:fprime}.



As a potential refinement of the bound $\fprime$, we also investigated the case where the hypercube contains a vertex whose incident edges are all of the same color. The hope was that this structural constraint might reduce the average number of color changes. However, this approach did not yield any improvement since for $n = 6$ the resulting value coincides with $\fprime(6)$.

\newcommand{\fenc}{F}
\newcommand{\fencprime}{\hat{\fenc}}
\subsection{Encoding the $f$ and $\fprime $ functions} \label{sec:fprime}

We extend the basic encoding from~\Cref{sec:enc} to allow computing the values $f(n),\fprime (n)$. More precisely, we construct formulas $\fenc(n, \alpha ),\fencprime(n,\alpha)$ which are satisfiable if and only if $f(n)  \geq \alpha, \fprime (n) \geq \alpha$ respectively. The main difference to the basic encoding is the necessity to count the number of color changes needed to reach another vertex and sum up over all antipodal vertex pairs.
For the sake of simplicity, we start with the description of $\fenc(n,\alpha)$ and later explain how to adapt it for $\fencprime(n,\alpha)$.
We use variables $p^{x}_{u,v,i}$ for $u,v \in V(Q_n), u \neq v, u \lexprec \antipodal{u}, x \in \{\text{red}, \text{blue} \}, i \in \{0, \ldots, n-1\}$ indicating whether there is a path from $u$ to $v$, whose last edge (i.e., closest to $v$) is colored $x$, with at most $i$ color changes.

We add the following constraints to ensure that the variables are set to true if a suitable path is present for $u,v,w \in V(Q_n), i \in \{0,\ldots, n-1\}, w \notin N(u), \dist{u}{v} + 1 = \dist{u}{w}, x \in \{\text{blue}, \text{red}\}$:
 \begin{align}
    r_{u,v} &\to p^{\text{red}}_{u,v, 0},               & \quad v \in N(u) \\
    \lnot r_{u, v} &\to p^{\text{blue}}_{u,v, 0},  & \quad v \in N(u) \\
    (p^{\text{red}}_{u,v,i} \land r_{v,w}) &\to p^{\text{red}}_{u,w,i},         & \quad v \in N(w) \label{eq:red_eq} \\
    (p^{\text{blue}}_{u,v,i} \land \lnot r_{v,w}) &\to p^{\text{blue}}_{u,w,i}, & \quad v \in N(w) \label{eq:blue_eq}\\
    (p^{\text{red}}_{u,v,i-1} \land \lnot r_{v,w}) &\to p^{\text{blue}}_{u,w,i}, & \; v \in N(w) \label{eq:red_swap}, i > 0 \\
    (p^{\text{blue}}_{u,v,i-1} \land r_{v,w}) &\to p^{\text{red}}_{u,w,i}, & \; v \in N(w) \label{eq:blue_swap}, i > 0 \\
    p^{x}_{u,v,i-1} &\to p^{x}_{u,v,i},               & i > 0
\end{align}
Note that we need  $\dist{u}{v} + 1 = \dist{u}{w}$ so we only consider paths along geodesics. Constraint~\eqref{eq:red_eq}~and~\eqref{eq:blue_eq} ensure that the count does not increase if the geodesic can be expanded by an edge with the same color, and Constraint~\eqref{eq:red_swap}~and~\eqref{eq:blue_swap} ensure that it increases by at most 1 if we pick a different color.

The minimum between two antipodal vertices is captured by $p^t_{u,i}$ variables:
    \begin{equation}
        p^t_{u, i} \leftrightarrow (p^{\text{red}}_{u, \bar{u}, i} \lor p^{\text{blue}}_{u, \bar{u}, i}),
    \end{equation}
    where actually only the $\leftarrow$ direction is needed, i.e., $p^{x}_{u, \bar{u}, i} \rightarrow p^t_{u, i} $  for $x \in \{\text{blue}, \text{red}\}$.

Note that for each pair $u, \antipodal{u}$ the number of color changes ($s_{u,\antipodal{u}, c}$ from Equation~\eqref{eq:fn}) is given by the number of variables $p^t_{u, i}$ set to false, i.e.,
\[s_{u, \antipodal{u}, c} = \sum_{i = 0}^{n-1} \lnot p^t_{u, i} .\]
For example if the number of changes is 2, then $p^t_{u, 0}$ and $p^t_{u, 1}$ are false. I.e., we can use cardinality constraints (we tried the \emph{modulo totalizer}~\citep{mtotalizer} and~\citet{Sinz05} with similar results)  over the negated literals to
bound the total number of color changes for all $2^{n-1}$ pairs of antipodal vertices:
    \begin{equation}
    \label{eq:f-count}
    \sum_{\substack{u \in V(Q_n)\\ {u \lexprec \bar{u}}}} \sum_{i = 0}^{n-1} \lnot p^t_{u, i} \geq 2^{n-1} \alpha.
    \end{equation}

For computing $\fprime $, we have to adapt slightly. 
We keep all constraints from $\fenc(n,\alpha)$ except~\eqref{eq:f-count}. Instead, we introduce the variables $p^t_{u, -1}$ and add  the constraints
\begin{align}
    (p^{\text{red}}_{u, \antipodal{u}, i}  \land  p^{\text{blue}}_{u, \antipodal{u}, i} )\rightarrow p^t_{u, i - 1} & \qquad i \in \{0,\ldots,n - 1\}
\end{align}
This change ensures that the count is decreased by one in cases where the second term, $ s'_{c, u, \antipodal{u}} - 1 $, determines the minimum in $ \min(s_{c, u, \antipodal{u}}, s'_{c, u, \antipodal{u}} - 1) $.

We have to include the $p^t_{u, -1}$ into the count and we have to slightly shift the count because we cannot ``subtract'' in a cardinality constraint. We do this by adding $+1$ for each antipodal vertex pair to the right side of the inequality and sum over $s$ starting with $-1$:
\begin{equation}
    \sum_{\substack{u \in V(Q_n)\\ u\lexprec\bar{u}}} \sum_{i = -1}^{n-1} \lnot p^t_{u, i} \geq 2^{n-1}  \alpha + 2^{n-1} 
    \end{equation}

We can also use a similar encoding for verifying~\Cref{conj:norine-allcolorings-geodesic} by adding the unit clauses $\lnot p_{u,\antipodal{u},1}^{x}$
for $u \in V(Q_n)$ and $x \in \{\text{red}, \text{blue} \}$.

In summary, we have a versatile encoding which allows us to test different conjectures and compute certain values for small values of $n$ by minor modifications of the encoding.

\subsection{Counting blocking pairs}

An interesting heuristic measure loosely tied to our conjectures is the maximum number of antipodal pairs $u, \bar{u}$ such that every geodesic between them has at least $2$ color changes. We call such pairs \emph{blocking}. We define $\mu(n)$ as the maximum number of blocking pairs, i.e.,
\[
\mu(n) = \max_{c \in C} \sum_{\substack{u \in Q(n) \\ u \lexprec \bar{u}}} [s_{c, u, \bar{u}} > 1].
\]
Note that using this notation,~\Cref{conj:norine-allcolorings-geodesic} is equivalent to whether $\mu(n) < 2^{n-1}$.

With a simple modification of the encoding $\Psi(n,\alpha)$ in~\Cref{sec:fprime} by replacing the constraint~\eqref{eq:f-count} with
\[ \sum_{\substack{u \in V(Q_n)\\ {u \lexprec \bar{u}}}} \lnot p^t_{u,1}, \]
we obtain the values and bounds presented in~\Cref{tab:table_mu}.

\begin{table}[h]
    \caption{Values of $\mu(n)$.}
    \centering
    \begin{tabular}{@{}crrrrrrr@{}}
        \toprule
        $n$ & 2 & 3 & 4 & 5 & 6 & 7 \\
        \midrule
        $\mu(n)$ & 0 & 1 & 2 & 6 & 14 & $\geq 29$ \\
        \bottomrule
    \end{tabular}
     \label{tab:table_mu}
\end{table}

In fact, it is possible to construct a specific $2$-edge-coloring of $Q_n$ for any $n$ that produces $2^{n-1}(1 - o(1))$ many blocking pairs (and matches the values for odd $n$ in~\Cref{tab:table_mu}).
The coloring divides the vertices of the hypercube into layers based on the number of $1$s, and alternates colors on edges between the layers.
Formally, it is $c_n : E(Q_n) \to \{0,1\}$ defined by:
\[
\left\{ (x_1, \dots, x_i, \dots, x_n), (x_1, \dots, 1 - x_i, \dots, x_n) \right\} \mapsto \bigoplus_{j \neq i} x_j,
\]
where the edge corresponds to flipping the $i$-th bit of a vertex, and the color is the parity of the sum of all other bits. An illustration for $Q_4$ is given in \Cref{fig:coloring_Q4}.

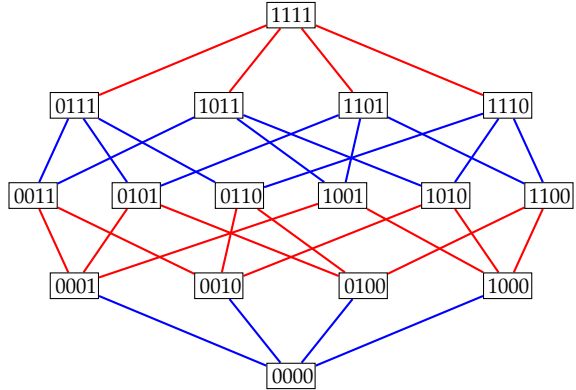
\begin{figure}[ht]
    \centering
    \begin{tikzpicture}[scale=0.6]
    \tikzset{every node/.style={text width=1.3em, inner sep=2pt, text centered, font=\scriptsize}}
\node[fill=white, draw] (0000) at (8.0,0) {0000};
\node[fill=white, draw] (0001) at (3.2,2) {0001};
\node[fill=white, draw] (0010) at (6.4,2) {0010};
\node[fill=white, draw] (0100) at (9.6,2) {0100};
\node[fill=white, draw] (1000) at (12.8,2) {1000};
[(0, 1), (0, 2), (0, 3), (1, 2), (1, 3), (2, 3)]
\node[fill=white, draw] (0011) at (2.2857142857142856,4) {0011};
\node[fill=white, draw] (0101) at (4.571428571428571,4) {0101};
\node[fill=white, draw] (0110) at (6.857142857142857,4) {0110};
\node[fill=white, draw] (1001) at (9.142857142857142,4) {1001};
\node[fill=white, draw] (1010) at (11.428571428571429,4) {1010};
\node[fill=white, draw] (1100) at (13.714285714285714,4) {1100};
\node[fill=white, draw] (0111) at (3.2,6) {0111};
\node[fill=white, draw] (1011) at (6.4,6) {1011};
\node[fill=white, draw] (1101) at (9.6,6) {1101};
\node[fill=white, draw] (1110) at (12.8,6) {1110};
\node[fill=white, draw] (1111) at (8.0,8) {1111};
\begin{pgfonlayer}{background}
\draw[opacity=1,  thick, blue] (0000) -- (0001);
\draw[opacity=1,  thick, blue] (0000) -- (0010);
\draw[opacity=1,  thick, blue] (0000) -- (0100);
\draw[opacity=1,  thick, blue] (0000) -- (1000);
\draw[opacity=1,  thick, red] (0001) -- (0011);
\draw[opacity=1,  thick, red] (0001) -- (0101);
\draw[opacity=1,  thick, red] (0001) -- (1001);
\draw[opacity=1,  thick, red] (0010) -- (0011);
\draw[opacity=1,  thick, red] (0010) -- (0110);
\draw[opacity=1,  thick, red] (0010) -- (1010);
\draw[opacity=1,  thick, red] (0100) -- (0101);
\draw[opacity=1,  thick, red] (0100) -- (0110);
\draw[opacity=1,  thick, red] (0100) -- (1100);
\draw[opacity=1,  thick, red] (1000) -- (1001);
\draw[opacity=1,  thick, red] (1000) -- (1010);
\draw[opacity=1,  thick, red] (1000) -- (1100);
\draw[opacity=1,  thick, blue] (0011) -- (0111);
\draw[opacity=1,  thick, blue] (0011) -- (1011);
\draw[opacity=1,  thick, blue] (0101) -- (0111);
\draw[opacity=1,  thick, blue] (0101) -- (1101);
\draw[opacity=1,  thick, blue] (0110) -- (0111);
\draw[opacity=1,  thick, blue] (0110) -- (1110);
\draw[opacity=1,  thick, blue] (1001) -- (1011);
\draw[opacity=1,  thick, blue] (1001) -- (1101);
\draw[opacity=1,  thick, blue] (1010) -- (1011);
\draw[opacity=1,  thick, blue] (1010) -- (1110);
\draw[opacity=1,  thick, blue] (1100) -- (1101);
\draw[opacity=1,  thick, blue] (1100) -- (1110);
\draw[opacity=1,  thick, red] (0111) -- (1111);
\draw[opacity=1,  thick, red] (1011) -- (1111);
\draw[opacity=1,  thick, red] (1101) -- (1111);
\draw[opacity=1,  thick, red] (1110) -- (1111);
\end{pgfonlayer}

    \end{tikzpicture}
    \caption{Edge coloring $c_4$ for the hypercube $Q_4$.}
    \label{fig:coloring_Q4}
\end{figure}

The coloring $c_n$ also yields a lower bound $\Omega(\sqrt{n})$ on the average number of color changes along geodesics (proofs of these bounds are in~\Cref{appendix:alternating}).
This provides a bound on how close to the original conjectures the approach based on $f$ and $\fprime$ can come: it can at best prove the existence of a path with $O(\sqrt{n})$ color changes.


\section{Experimental results} \label{sec:experiments}

\paragraph{Hardware and Solvers.} We used a cluster with two AMD EPYC 7742 CPUs, each with
64 cores of
2.25-3.40GHz, running Debian. Depending on the experiment, we used the award-winning solvers \texttt{kissat}~\citep{SAT-Competition-2024-solvers} and \texttt{cadical}~\citep{BiereFallerFazekasFleuryFroleyks24}, together with the splitting tools \texttt{march\_cu}~\citep{HeuleKullmannWieringaBiere11} and \texttt{proofix}~\citep{proofix}.

\paragraph{Symmetry Breaking.}
We first evaluate the impact of the \texttt{max\_comp} parameter, which limits the number of elements compared in a lex-leader constraint. \Cref{fig:sym-param} shows the running time for proving~\Cref{conj:norine} for $n = 7$. For each value of the parameter we run 16 instances of $\Phi_7$ with \texttt{kissat} in which the variables and clauses are randomly shuffled using \texttt{scranfilize}~\citep{POS-18:Effect_Scrambling_CNFs}. The results indicate that the initial comparisons are the most effective in pruning the search space (without symmetry breaking it takes over an hour), with performance peaking at around 20 comparisons. Beyond this point, the running time increases gradually, since the number of clauses keeps increasing without breaking many symmetries.
Note that for different values of $n$, the optimal number of comparisons may vary, although we expect similar behavior. Based on these observations, we chose to allow slightly more comparisons to ensure robust performance. Furthermore, as a form of ablation testing, we ran the encoding of~\citet{frankston2024provingnorinesconjectureholds} on $n=7$ replacing their symmetry breaking with ours, which reduced the runtime from $4$ minutes to $1$ minute. The results for $\Psi_7$ are almost identical, but with a runtime increase of $70\%$  on average. 
As described in~\Cref{subsec:symmetry-breaking}, the lexicographic symmetry breaking is compatible with vertex $\vec{0}$ having the minimum \emph{red degree}, meaning that no other vertex can be incident to more red edges. If the lexicographic symmetry breaking was \emph{complete}, then it would actually enforce that vertex $\vec{0}$ has the minimum red degree, but as we perform incomplete symmetry breaking, 
it is still helpful to encode this constraint explicitly with a cardinality constraint. For example, the verification of~\Cref{conj:norine-allcolorings} for $n = 6$ goes down from $32$ seconds to $25$ seconds by the addition of this constraint.

\begin{figure}
    \begin{subfigure}{0.49\linewidth}
        \centering
\begin{tikzpicture}[scale=0.7]
  \begin{axis}[ylabel={Time (s)},xlabel={Symmetry-breaking parameter \texttt{max\_comp}},legend style={at={(0.45,0.97)}, legend cell align=left}, grid=major, grid style={gray!60!white}, ytick={20, 30, 40, 50, 60},
  xmin=0,
  xmax=51,
  ymax=71]
\addplot[color=magenta,mark=x]
coordinates {
(0, 72)
(1, 72)
(2, 65.61)
(3, 64.9)
(4, 48.79)
(5, 41.04)
(6, 40.15)
(7, 38.22)
(8, 36.4)
(9, 34.84)
(10, 36.93)
(11, 36.26)
(12, 37.1)
(13, 37.87)
(14, 37.45)
(15, 33.68)
(16, 37.17)
(17, 38.44)
(18, 38.3)
(19, 33.27)
(20, 34.43)
(21, 36.06)
(22, 37.17)
(23, 37.63)
(24, 39.84)
(25, 37.81)
(26, 36.68)
(27, 38.04)
(28, 36.59)
(29, 36.05)
(30, 45.18)
(31, 38.94)
(32, 40.34)
(33, 36.91)
(34, 41.16)
(35, 36.34)
(36, 40.42)
(37, 37.6)
(38, 40.3)
(39, 39.48)
(40, 44.69)
(41, 42.58)
(42, 37.01)
(43, 39.67)
(44, 42.36)
(45, 37.78)
(46, 42.24)
(47, 42.21)
(48, 39.62)
(49, 43.39)
(50, 45.06)
};

\addplot[color=cyan!60!black,mark=x]
coordinates {
(0, 72)
(1, 72)
(2, 55.724375)
(3, 47.945)
(4, 39.113125000000004)
(5, 33.893125)
(6, 31.6525)
(7, 32.965)
(8, 30.038124999999994)
(9, 28.791874999999994)
(10, 29.816875)
(11, 28.807500000000005)
(12, 29.35)
(13, 29.94375)
(14, 28.978125000000002)
(15, 28.9025)
(16, 29.517500000000002)
(17, 30.046874999999996)
(18, 29.2475)
(19, 28.919999999999995)
(20, 29.113125)
(21, 29.80375)
(22, 31.193749999999998)
(23, 29.806874999999998)
(24, 32.46187499999999)
(25, 30.613125)
(26, 31.54875)
(27, 30.368125)
(28, 31.939999999999998)
(29, 29.7275)
(30, 31.390625)
(31, 32.629999999999995)
(32, 31.8925)
(33, 29.909375)
(34, 32.862500000000004)
(35, 31.930624999999996)
(36, 32.317499999999995)
(37, 31.979374999999994)
(38, 33.136250000000004)
(39, 32.7275)
(40, 34.21562500000001)
(41, 33.645624999999995)
(42, 31.2625)
(43, 31.9025)
(44, 35.513124999999995)
(45, 31.249374999999997)
(46, 33.46)
(47, 33.355624999999996)
(48, 32.53)
(49, 34.178125)
(50, 35.608125)
};

\addplot[color=orange,mark=x]
coordinates {
(0, 72)
(1, 72)
(2, 38.62)
(3, 33.41)
(4, 27.35)
(5, 26.53)
(6, 27.73)
(7, 26.18)
(8, 24.5)
(9, 23.58)
(10, 24.62)
(11, 22.6)
(12, 23.52)
(13, 24.09)
(14, 26.06)
(15, 23.97)
(16, 22.71)
(17, 24.27)
(18, 22.35)
(19, 22.69)
(20, 22.6)
(21, 24.56)
(22, 24.95)
(23, 24.57)
(24, 26.62)
(25, 22.09)
(26, 26.09)
(27, 23.61)
(28, 27.24)
(29, 25.99)
(30, 20.59)
(31, 27.23)
(32, 20.58)
(33, 22.67)
(34, 27.53)
(35, 27.41)
(36, 25.42)
(37, 27.15)
(38, 26.4)
(39, 25.37)
(40, 25.36)
(41, 26.89)
(42, 21.67)
(43, 21.95)
(44, 28.51)
(45, 24.91)
(46, 25.34)
(47, 23.26)
(48, 25.66)
(49, 25.96)
(50, 28.64)
};
\legend{Max\@16,  Avg\@16, Min\@16}
\end{axis}

\begin{axis}[legend style={at={(0.85,0.97)}, legend cell align=left}, axis x line=none,          
ylabel={\#symmetry-breaking clauses},
  axis y line*=right,     xmin=0,
  xmax=51,
    scaled y ticks = false,
    ytick = {0, 10000,20000,30000},
    yticklabels = {0, 10K,20K,30K}]
    \addplot[color=green!50!black, thick]
coordinates {
(0, 0)
(1, 9247)
(2, 9751)
(3, 10255)
(4, 10759)
(5, 11263)
(6, 11767)
(7, 12271)
(8, 12775)
(9, 13279)
(10, 13783)
(11, 14287)
(12, 14791)
(13, 15295)
(14, 15799)
(15, 16303)
(16, 16807)
(17, 17311)
(18, 17815)
(19, 18319)
(20, 18823)
(21, 19327)
(22, 19831)
(23, 20335)
(24, 20839)
(25, 21343)
(26, 21847)
(27, 22351)
(28, 22855)
(29, 23359)
(30, 23863)
(31, 24367)
(32, 24871)
(33, 25375)
(34, 25879)
(35, 26383)
(36, 26887)
(37, 27391)
(38, 27895)
(39, 28399)
(40, 28903)
(41, 29407)
(42, 29911)
(43, 30415)
(44, 30919)
(45, 31423)
(46, 31927)
(47, 32431)
(48, 32935)
(49, 33439)
(50, 33943)
};
\legend{sb clauses}
\end{axis}


\end{tikzpicture}
    \caption{Effect of symmetry-breaking parameter \texttt{max\_comp} on the runtime of $\Phi_7$.}
    \label{fig:sym-param}
    \end{subfigure}
    \begin{subfigure}{0.49\linewidth}
        \centering
    \input{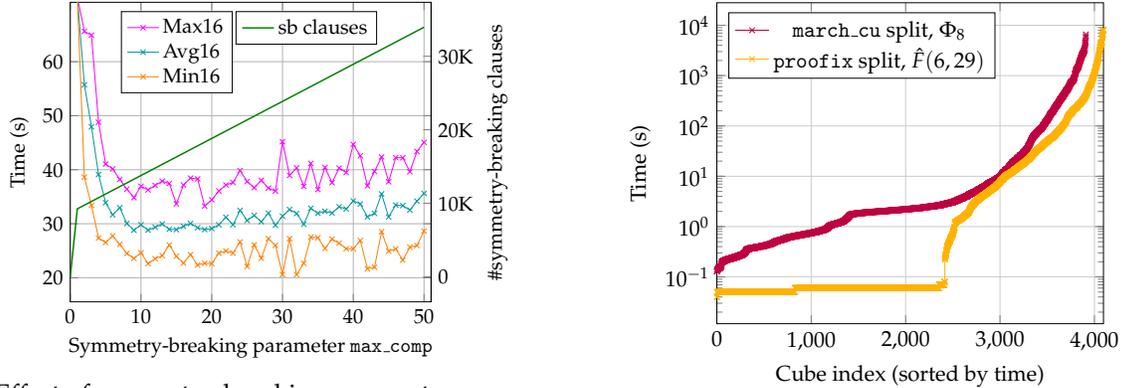}
    \caption{Cube And Conquer splits for $\Phi_8$ and $\hat{f}(6).$}
    \label{fig:cubing-plots}
    \end{subfigure}
    \caption{Experimental results.}
\end{figure}

\paragraph{Hard Instances.}

With the encoding from~\Cref{sec:enc}, and setting the symmetry-breaking parameter to $30$, we were able to prove the unsatisfiability of $\Phi_8$. More precisely, we first ran \texttt{kissat} with the \texttt{-o} option to simplify the formula, for $20$ million conflicts, and then used \texttt{march\_cu} with depth $12$, resulting in 3911 cubes. The cubes were then assigned to $64$ processors running~\texttt{cadical}, with the $i$-th cube being assigned to the $i \bmod 64$-th core. The total process runtime was 116.32 hours, and the total wall-clock time was 3.98 hours. The average time per cube was 107 seconds. 
The satisfiability of the formula $\fencprime(6, \nicefrac{28}{32})$, corresponding to the computation of $\hat{f}(6)$, is shown in a few seconds.
The unsatisfiable formula $\fencprime(6, \nicefrac{29}{32})$ was cubed similarly. 
The main difference is that since this formula involves a cardinality constraint, the splitting tool~\texttt{proofix} worked better, choosing $4$ out of $12$ splitting variables to be auxiliary variables of the~\emph{modulo totalizer} encoding~\citep{mtotalizer}. In this case, we generated 4096 cubes, which were completed in 124.46 hours of total process time, or 14.6 hours of wall-clock time. As can be appreciated in~\Cref{fig:cubing-plots}, the runtimes between the two formulas are very similar, with the split for $\fencprime(6, \nicefrac{29}{32})$ being slightly worse (many easy cubes), which resulted in a larger wall-clock time. In general, the formulas $\fencprime(n, \alpha)$ are significantly harder than $\Phi_n$, thus explaining why we can solve the latter for $n=8$ but the former only until $n=6$.

    

\Cref{tab:bounds}~summarizes the computed values of $f$ and $\fprime$. Notably, computing $\fprime(4)$ allows us to match the best previously known asymptotic bound. This highlights the importance of the refined approach that leads to $\fprime$: computing only the $f$-values does not suffice to improve upon the previous bound, as seen for $k=6$, where $f(6)$ fails to meet the threshold while $\fprime(6)$ does.

\begin{table}
\caption{
Threshold for improving \citeauthor{dvorakNoteNorinesAntipodalColouring2020}'s asymptotic result (2nd column) and the values of $f$ and $\hat{f}$.}
    \centering
    \begin{tabular}{@{}crr@{\;}lr@{}}
    \toprule
        $k$ & $\frac{3k}{8} - 1$ & $f(k)$ & & $\hat{f} (k)$ \\
        \midrule
        3 & 0.125 & 1 &\citep{dvorakNoteNorinesAntipodalColouring2020} & 0.5 \\
        4 & 0.5 & 1.25 &\citep{ReuSlides} &  0.5 \\
        5 & 0.875 & 1.25 & &  0.875 \\
        6 & 1.25& 1.5 & & \best{0.875} \\ 
        \bottomrule
    \end{tabular}   
    \label{tab:bounds}
\end{table}


\section{Conclusion}\label{sec:conclusion}

We have presented an encoding for Norin's conjecture that is substantially more compact than previous encodings, allowing us to verify the conjecture for $n=8$ in under $4$ hours of wall-clock time, even though~\citet{frankston2024provingnorinesconjectureholds} estimated 57 CPU years for this task. The versatility of our encoding allowed us to compute the average number of color changes in a worst-case coloring for $n = 6$, which, combined with ideas from~\cite{dvorakNoteNorinesAntipodalColouring2020}, has allowed us to improve the asymptotic upper bound on the number of color changes to $0.3125n + O(1)$. \cite{dvorakNoteNorinesAntipodalColouring2020} presented a careful probabilistic analysis, allowing the bound of $(\nicefrac{3}{8} + o(1))n$ instead of the naive $(\nicefrac{2}{3} + o(1))n$ bound one would obtain directly from $f(3) = 1$, whereas the analysis required for our $\hat{f}(6)$-approach is rougher, and thus more scalable.

Several avenues are open for future work. First, while we believe verifying the conjecture for $n = 9$ will require further optimizations, the resulting formula has $2.5$ million clauses, which is within the limits of modern SAT solving. Furthermore, the formula for $\fprime(7)$  has fewer than a million clauses, making it a more immediate candidate for future research. More generally, we hope that our encoding can be extended to other computations for hypercubes of small dimension, which could provide more insight into the conjecture.

\bibliography{references}
\bibliographystyle{plainnat}

\appendix

\section{Alternating Coloring}
\label{appendix:alternating}

In the ``alternating coloring'' of the hypercube, 
where an edge corresponds to flipping the $i$-th bit of a vertex for some $i$, the color of that edge is given by the parity of the sum of all other bits.

\begin{lemma}
    Let $n \in \mathbb{N}$ be odd, let $c : E(Q_n) \to \{0,1\}$ be defined by:
    \[
        ((x_1, \dots, x_i, \dots, x_n),\, (x_1, \dots, 1 - x_i, \dots, x_n)) \mapsto \bigoplus_{j \neq i} x_j.
    \]
    The number of blocking antipodal pairs in $c$ is $2^{n-1}\left(1-O\left(\frac{1}{\sqrt{n}}\right)\right)$
\end{lemma}
\begin{proof}
Consider a path $u\rightarrow v \rightarrow w$, and suppose $u$ and $v$ differ at position $i$ and $v$ and $w$ at position $j$.
The color of the edge $uv$ is $c_{uv} = \bigoplus_{l\not\in \{i, j\}} u_l \; \oplus u_j$.
The color of the edge $vw$ is $c_{vw} = \bigoplus_{l\not\in \{i, j\}} v_l \; \oplus v_i$.
We have $c_{uv} \oplus c_{vw} = u_j \oplus v_i = u_j \oplus (1-u_i) = 1 \oplus u_j \oplus u_i$.
So, there is a color change at $v$ if and only if $u_i = u_j$, i.e., the two consecutive coordinates both flip $0$ to $1$ or both $1$ to $0$.

An antipodal geodesic from some vertex $v$ must flip all coordinates one by one.
If the numbers of $0$s and $1$s in $v$ differ by at most one, the flips can be interleaved in such a way that no color changes are needed.
If they differ by $2$, it can be done with one color change.
If they differ by at least $3$, two color changes are needed.

Let $\beta(v)$ be the difference between the number of $1$s and $0$s in $v$, namely $\beta(v) = \sum v_i - \sum (1-v_i) = 2\sum v_i - n = \sum v_i - \sum \bar{v}_i = -\beta(\bar v)$.
Represent an antipodal pair by the vertex which has more $0$s than $1$s, i.e., where $\beta(v) < 0$ (recall $n$ is odd).
We want to count the number of $v$ for which $\beta(v) \leq -3$.
We have $\beta(v) = 2\sum v_i - n \leq -3$ if and only if $\sum v_i \leq \frac{n-3}{2}$.
The number of such $v$ is
\[ g(n) = \sum_{i = 0}^{\frac{n-3}{2}} \binom{n}{i} .\]

Write $n=2k+1$, let $h(k) = g(2k+1)$, and consider that

\begin{align*}
    2^{2k+1} = 2^n = \sum_{i=0}^{n} \binom{n}{i} = 
    2h(k) + \binom{2k+1}{k} + \binom{2k+1}{k+1} = 
    2\left(h(k) + \binom{2k+1}{k}\right),
\end{align*} 

and thus

\[ h(k) = 2^{2k} - \binom{2k+1}{k} .\]
Plug in
\[ \binom{2k+1}{k} = \binom{2k}{k} + \binom{2k}{k-1} = \binom{2k}{k} \left(1 + \frac{k}{k+1}\right),\]
and apply the upper bound on the central binomial coefficient that arises from Stirling's approximation,
\( \binom{2k}{k} \leq \frac{2^{2k}}{\sqrt{\pi k}}\),
to get
\begin{align*}
\binom{2k+1}{k} \leq \left(1 + \frac{k}{k+1}\right)\frac{2^{2k}}{\sqrt{\pi k}} =
O \left( 2^{n-1} \frac{1}{\sqrt{n}} \right).
\end{align*}

\end{proof}

For even $n$, it is possible to copy the construction for odd $n-1$, ignoring the last coordinate completely (and coloring edges that flip it arbitrarily).
This yields $2g(n-1)$ bad pairs for even $n$, and the asymptotic result is the same.

It turns out that the alternating coloring also shows that the $f$-based approach cannot prove a result better than $c \cdot \sqrt{n}$ for some absolute constant $c$.

\begin{lemma}\label{lemma:color_changes_weight}
    For any vertex $v \in V(Q_n)$, the number of color changes from $v$ to $\bar{v}$ in any path $\pi$ in the coloring $c$ defined above is at least $|\beta(v)| - 1$.
\end{lemma}
\begin{proof}
Let $\pi = \pi_1, \pi_2, \pi_3, \ldots, \pi_{m}$, with $\pi_1 := v$ and $\pi_m := \bar{v}$, for some $m \geq n$. 
    For ease of notation, let us define 
    \(
        w_i := \beta(\pi_i)
    \). 
    Then, let $\Delta_i := w_{i+1} - w_i$ for $1 \leq i \leq m-1$, and note that
    \[
 |w_{i+2} - w_{i}| = \begin{cases}
         4 & \text{if }\pi_{i}, \pi_{i+1}, \pi_{i+2} \text{ is a color change}\\
         0 & \text{otherwise}.
     \end{cases} 
    \]
    Now, observe that
    \begin{align*}
    \left|\sum_{i=1}^{m-2} w_{i+2} - w_{i}\right| &= |w_m + w_{m-1} - (w_2 + w_1)|\\
    &= |2w_m - \Delta_m - 2w_1 - \Delta_2|\\
    &\geq 2|w_m - w_1| - |\Delta_m + \Delta_2|\\
    &\geq 2|w_m - w_1| - 4 \tag{as $|\Delta_i| \leq 2$}\\
    &= 4|w_1| - 4. \tag{as $w_m = -w_1$}
    \end{align*}
   Thus, there must be at least $|w_1| - 1$ terms of the sum $\sum_{i=1}^{m-2} w_{i+2} - w_{i}$ whose absolute value is at least $4$, which implies at least $|w_1|-1$ color changes.
\end{proof}

\begin{lemma}
    There exists some $\alpha > 0$ such that for any sufficiently large $k$,  we have $f(k) \geq \alpha \cdot \sqrt{k}$. In particular, $\alpha =0.7$ suffices.
\end{lemma}
\begin{proof}
    We consider the same coloring $c$ as above. Let us denote by $\mathcal{P}(v)$ the set of all paths between a vertex $v$ and its antipodal $\bar{v}$, and use notation $s(\pi, c)$ to denote the number of color changes of $\pi$ under coloring $c$. Thus,
    \begin{align*}
    f(k) &= \frac{1}{2^k} \sum\limits_{v \in V(Q_k)}  \min\limits_{\pi \in \mathcal{P}(v)} s(\pi, c)\\
    &=  \frac{1}{2^k} \sum_{\ell = 0}^{k} \sum_{\|v\|_1 = \ell} \min\limits_{\pi \in \mathcal{P}(v)} s(\pi, c)\\
     &\geq  \frac{1}{2^k} \sum_{\ell = 0}^{k}  |k - 2\ell| \binom{k}{\ell} - 1. \tag{Using~\Cref{lemma:color_changes_weight}}
    \end{align*}

Let us assume that $k = 2m+1$, and the even case will be a trivial consequence.
We can analyze the sum $\sum_{\ell = 0}^{k}  |k - 2\ell| \binom{k}{\ell}$  by using its symmetry around $m$. Indeed,
\[
\sum_{\ell = 0}^{k}  |k - 2\ell| \binom{k}{\ell} = 2\sum_{\ell =0}^{m} (k - 2\ell) \binom{k}{\ell},
\]
and 
\begin{align*}
(k-2\ell) \binom{k}{\ell} = (k-\ell)\binom{k}{k -\ell} - \ell \binom{k}{\ell}
= \frac{k!}{(k-\ell - 1)! \cdot \ell !} - \frac{k!}{(\ell-1)! (k-\ell)!}
= k \binom{k-1}{\ell} - k \binom{k-1}{\ell - 1},
\end{align*}
where we have adopted the convention $(-1)! = 1, \binom{k-1}{-1} = 0.$
Therefore,
\begin{align*}
\sum_{\ell = 0}^{k}  |k - 2\ell| \binom{k}{\ell} = 2k\sum_{\ell=0}^m \binom{k-1}{\ell} - \binom{k-1}{\ell - 1} 
= 2k\binom{k-1}{m},
\end{align*}
where the last equality is a telescopic argument, using again $\binom{k-1}{-1} = 0$.
To conclude, note that 
\[
2k\binom{k-1}{m} = 2k\binom{2m}{m} > 2k \frac{4^m}{\sqrt{4m}} = \frac{\sqrt{k} 2^{k}}{\sqrt{2}},
\]
and thus 
\[
  f(k) > \frac{1}{2^k}\left(\frac{\sqrt{k} 2^{k}}{\sqrt{2}} - (k+1)\right) = \sqrt{\nicefrac{1}{2}} \cdot \sqrt{k} - \frac{k+1}{2^k}.
\]
For $k \geq 5$ we get
$f(k) \geq 0.7 \sqrt{k}$.
For the even $k$ case, it suffices to use that $f$ is non-decreasing~\citep{ReuSlides}, as then
\begin{align*}
    f(k+1) &\geq  f(k)\\
    &\geq \sqrt{\frac{1}{2}} \cdot \sqrt{k} - \frac{k+1}{2^k}\\ 
    &= \sqrt{\frac{k}{2(k+1)}} \sqrt{k+1} - \frac{k+1}{2^k},
\end{align*}

and $ \sqrt{\frac{k}{2(k+1)}} - \frac{k+1}{2^k} > 0.7$ for $k \geq 25$, which yields that $f(k+1) > 0.7\sqrt{k+1}$.    
\end{proof}

\end{document}